\documentclass[10pt,reqno]{amsart}
\usepackage{amsmath}
\usepackage{amscd,amsthm,amsfonts,amsopn,amssymb}
\usepackage{epsfig,hyperref}

\newtheorem{theorem}{Theorem}[section]

\newtheorem{proposition}[theorem]{Proposition}
\newtheorem{corollary}[theorem]{Corollary}

\newtheorem{lemma}[theorem]{Lemma}
\theoremstyle{remark}
\newtheorem{remark}[theorem]{Remark}

\DeclareMathOperator{\esssup}{esssup\,}

\DeclareMathOperator{\supp}{supp\,}

\DeclareMathOperator{\dist}{dist\,}

\DeclareMathOperator{\divop}{div}
\DeclareMathOperator{\iRe}{Re}
\DeclareMathOperator{\iIm}{Im}

\def\XXint#1#2#3{{\setbox0=\hbox{$#1{#2#3}{\int}$ }
\vcenter{\hbox{$#2#3$ }}\kern-.6\wd0}}

\allowdisplaybreaks
\begin{document}

\title[Blow-up criteria below scaling]{Blow-up criteria below scaling for defocusing energy-supercritical NLS and quantitative global scattering bounds}

\author[Aynur Bulut]{Aynur Bulut}
\address{Department of Mathematics, Louisiana State University, 303 Lockett Hall, Baton Rouge, LA 70803}
\email{aynurbulut@lsu.edu}

\dedicatory{The author dedicates this work to the memory of Jean Bourgain.}

\begin{abstract}
We establish quantitative blow-up criteria below the scaling threshold for radially symmetric solutions to the defocusing nonlinear Schr\"odinger
equation with nonlinearity $|u|^6u$.  This provides to our knowledge the first generic results distinguishing potential blow-up solutions of the 
defocusing equation from many of the known examples of blow-up in the focusing case.  Our main tool is a quantitative version of a result showing 
that uniform bounds on $L^2$-based critical Sobolev norms imply scattering estimates.  \\

As another application of our techniques, we establish a variant which allows for slow growth in the critical norm.  We show that if the critical 
Sobolev norm on compact time intervals is controlled by a slowly growing quantity depending on the Stricharz norm, then the solution can be 
extended globally in time, with a corresponding scattering estimate.
\end{abstract}

\maketitle

\section{Introduction}

We consider an energy-supercritical instance of the defocusing Nonlinear Schr\"o\-dinger (NLS) equation on $\mathbb{R}^3$,
\begin{align}
\label{eq1}\left\lbrace \begin{array}{rl}iu_t+\Delta u&=|u|^6u,\quad (t,x)\in I\times\mathbb{R}^3,\\
u|_{t=0}&=u_0\in \dot{H}_x^{s_c},\end{array}\right.
\end{align}
with radially symmetric initial data in the scaling critical Sobolev space $\dot{H}^{s_c}(\mathbb{R}^3)$, $s_c=7/6$, posed on a time 
interval $0\in I\subset\mathbb{R}$.  Here, $\dot{H}^{s_c}$ denotes the usual homogeneous Sobolev space, with norm given by
\begin{align*}
\lVert f\rVert_{\dot{H}^{s_c}}^2=\int |\xi|^{2s_c}|\hat{f}(\xi)|^2d\xi,
\end{align*}
where $\hat{f}$ denotes the Fourier transform of $f$.

It is well known (see, for instance, \cite{C,KVnotes} and \cite{Murphy}) that initial value problems of type (\ref{eq1}) are equipped with a robust
local well-posedness theory whenever the initial data belongs to $\dot{H}_x^{s}(\mathbb{R}^3)$ with $s\geq s_c$, where the space 
$\dot{H}_x^{s_c}(\mathbb{R}^3)$, $s_c=7/6$ is distinguished as the only homogeneous $L^2$-based Sobolev space preserved by the scaling 
symmetry of the equation,
$$u(t,x)\leadsto \lambda^{-1/3}u(\lambda^{-2}t,\lambda^{-1}x).$$

On the other hand, the status of global-in-time well-posedness properties of solutions to (\ref{eq1}) is largely unresolved, even for
radially symmetric smooth initial data with compact support.  While the mass 
$$M[u(t)]:=\int_{\mathbb{R}^3} |u(t)|^2dx,$$ and energy, $$E[u(t)]:=\frac{1}{2}\int_{\mathbb{R}^3} |\nabla_x u(t)|^2dx+\frac{1}{8}\int |u(t)|^8dx,$$
are preserved by the evolution (that is, $M[u(t)]=M[u_0]$ and $E[u(t)]=E[u_0]$ hold for all $t\in I$), when the nonlinearity is in the energy-supercritical 
regime as in (\ref{eq1}) these quantities have subcritical scaling, and do not provide sufficient control to allow for iteration of the local theory.

Nevertheless, some partial results concerning possible blow-up scenarios for defocusing energy-supercritical nonlinearities are known.  First steps in 
this direction are the uniform {\it a priori} critical bound implies scattering results of \cite{KVscNLS} and \cite{MJZ,Murphy}, based on concentration 
compactness-based methods developed in the 
context of the nonlinear Schr\"odinger and nonlinear wave equations (see, for instance \cite{KM-NLS,KM-NLW,KM-NLWsc} and \cite{KV-focusing-critical-NLS,KVscNLW3,KVscNLWradial,Bsc1,Bsc2,Bsc3}, as 
well as the references cited in these works).  These results show that that solutions cannot blow-up in finite time whenever
\begin{align*}
E_{s_c}:=\lVert u(t)\rVert_{L_t^\infty(I_{\textrm{max}};\dot{H}_x^{s_c})}
\end{align*}
is finite, with $I_{\max}\subset\mathbb{R}$ denoting the maximal interval of existence.

Other results concerning possible blow-up in the energy-supercritical regime include studies of global well-posedness for logarithmically supercritical 
nonlinearities \cite{Roy-0} (see also related works \cite{TaoLogNLW,RoyNLW,Roy,Shih,BuDo,CoHa} for the nonlinear wave equation), and constructions of large 
data global solutions \cite{BDSW}.  We remark that there are also well-developed constructions of blow-up solutions for focusing nonlinearities (where 
$|u|^6u$ is replaced by $-|u|^{p}u$ for $p>\frac{4}{d-2}$, with the problem posed on $\mathbb{R}^d$); see, e.g. \cite{Collot,DS,MRR}, 
and the references cited therein.

Very recently, in a groundbreaking work \cite{MRRS}, Merle, Rapha\"el, Rodnianski, and Szeftel have constructed the first instance of finite-time blow-up for 
the defocusing energy-supercritical nonlinear Schr\"odinger equation (earlier blow-up results for systems were obtained by Tao in \cite{TaoBlowupNLW,TaoBlowupNLS}).
The construction in \cite{MRRS} is based on a reduction to the formation of shocks in an instance of the compressible Euler equation, and produces a blow-up
solution for which 
\begin{align}
\lim_{t\nearrow T_*} \lVert u(t)\rVert_{H_x^s}=\infty,\quad s_*<s<s_c,\label{bounded-hs}
\end{align}
for some $s_*\in (1,s_c)$, where $T_*$ is the blowup time.  This distinguishes the {\it defocusing} blow-up scenario constructed in \cite{MRRS} 
from the soliton-focusing and self-similar scenarios constructed in the focusing setting.  We remark that the current understanding of the hydrodynamic reduction 
is tied to the high-dimensional ($\mathbb{R}^d$ with $d\geq 5$) setting, and \cite{MRRS} therefore does not directly address the issue of blow-up in the 
three-dimensional setting of (\ref{eq1}).

Our first main result is that any potential blow-up solution to (\ref{eq1}) must exhibit growth of subcritical Sobolev norms as in the higher-dimensional 
blow-up scenario constructed in \cite{MRRS}.  In particular, we show that when the initial data $u_0$ has some additional regularity, 
e.g. $$u_0\in \dot{H}_x^{s_c}(\mathbb{R}^3)\cap \dot{H}_x^{s_c+1}(\mathbb{R}^3),$$ global well-posedness remains true under an assumption 
on ``slightly subcritical norms,''
\begin{align*}
\sup_{t\in I}\,\, \lVert u(t)\rVert_{\dot{H}_x^{s_c-\delta}(\mathbb{R}^3)}\leq E,
\end{align*}
for $\delta>0$ sufficiently small, with the required smallness depending on $\lVert u_0\rVert_{\dot{H}_x^{s_c+1}}$ and $E$.  The precise statement is given
in the theorem below.

\begin{theorem}
\label{thm1}
There exists $C>0$ such that for each $E\geq 1$ and $M>0$ there exists $\delta_0=\delta_0(E,M)>0$ with the following property: For all 
radially symmetric initial data $u_0\in \dot{H}_x^{s_c}(\mathbb{R}^3)\cap \dot{H}_x^{s_c+1}(\mathbb{R}^3)$, with
$$\lVert u_0\rVert_{\dot{H}_x^{s_c}(\mathbb{R}^3)\cap \dot{H}_x^{s_c+1}(\mathbb{R}^3)}\leq M,$$
if $0<\delta<\delta_0$ and $$u\in C_t(I;\dot{H}_x^{s_c}(\mathbb{R}^3))\cap L_{t,x}^{15}(I\times\mathbb{R}^3)\quad\textrm{for all}\quad I\subset\subset I_{\textrm{max}},$$ is a solution with maximal-lifespan $I_{\textrm{max}}$  
to (\ref{eq1}) which satisfies,
\begin{align}
\lVert u\rVert_{L_t^\infty(I_{\textrm{max}};\dot{H}_x^{s_c-\delta}(\mathbb{R}^3))}\leq E,\label{eq-E0}
\end{align}
then $I_{\textrm{max}}=\mathbb{R}$ and $$\lVert u\rVert_{L_{t,x}^{15}(\mathbb{R}\times\mathbb{R}^3)}\leq C\exp(C(EM^\delta)^C).$$
\end{theorem}

In this statement and the rest of this paper, a {\it solution} to (\ref{eq1}) on a time interval $I$ will always be understood as a function in 
$C_t(J;\dot{H}_x^{s_c}(\mathbb{R}^3))\cap L_{t,x}^{15}(J\times\mathbb{R}^3)$ for all intervals $J\subset\subset I$, 
which satisfies the initial value problem in the sense of the associated integral equation,
\begin{align*}
u(t)=e^{it\Delta}u_0-i\int_0^t e^{i(t-t')\Delta}[|u(t')|^{6}u(t')]dt'.
\end{align*}

We remark that the restriction to (\ref{eq1}) in our work---that is, considering the case of three spatial dimensions and $p=6$ in the 
nonlinearity $|u|^pu$---is not essential (we have chosen to work with this case to highlight the core aspects of the arguments; it 
is expected that standard techniques would allow to treat more general energy-supercritical defocusing problems, with the associated 
technical subtleties, for instance in treating non-smooth nonlinearities, well understood in related energy-critical settings, see, 
e.g. \cite{TV,V,KVZ}).  In light of this, Theorem $\ref{thm1}$ is a strong indication that properties such as those shown in \cite[Appendix D]{MRRS} 
are indeed {\it universal} properties of any defocusing energy-supercritical blow-up for NLS.

The first indication that a result like Theorem $\ref{thm1}$ should hold arises from \cite{Roy}, where Roy studied the nonlinear wave 
equation with log-supercritical defocusing nonlinearities of type $|u|^pug(|u|)$.  The main result of \cite{Roy} is that a robust global well-posedness and 
scattering theory for the energy-supercritical nonlinearity $|u|^pu$ with data in the critical space $\dot{H}^{s_c}\times\dot{H}^{s_c-1}$ would imply
that solutions to the log-supercritical problem which are uniformly bounded in $\dot{H}^{s_c}\times\dot{H}^{s_c-1}$ can be extended
globally in time.  A second motivation, accounting for the transition from uniform $\dot{H}_x^{s_c}$ bounds to uniform control in $\dot{H}_x^{s_c-\delta}$, is 
found in a recent paper of Colombo and Haffter \cite{CoHa}, also studying the nonlinear wave equation, where the authors revisit the log-supercritical 
theory of \cite{Roy-0}, replacing log-supercriticality by small power-type supercriticality for bounded sets of initial data, with the level of 
supercriticality depending on the initial-data bound.  The authors of \cite{CoHa} were motivated by recent results for supercritical problems in 
fluid dynamics, where supercriticality is compensated for in a similar way (see, e.g.  \cite{CZV} for the surface quasi-geostrophic equation and 
the associated log-supercritical result in \cite[Theorem 1.3]{BulutDong}, as well as \cite{CH} for the hyperdissipative Navier-Stokes system).

The main tool used in the proof of Theorem $\ref{thm1}$ is the following proposition, which shows that {\it a priori} uniform control on the 
critical norm implies an explicit quantitative bound on the scattering norm (and therefore uniform control over the critical $\dot{H}_x^{s_c}$ norm 
on a solution's maximal interval of existence leads to global well-posedness).  

\begin{proposition}
\label{prop-uniform}
Suppose that $u$ is a radially symmetric solution to (\ref{eq1}) which has maximal interval of existence $I_{\textrm{max}}$ and satisfies 
\begin{align}
\lVert u\rVert_{L_t^\infty(I;\dot{H}_x^{s_c}(\mathbb{R}^3))}\leq E=E_{s_c}<+\infty\label{eq-sc-uniform}
\end{align}
for some $I\subset I_{\textrm{max}}$.  Then 
$\lVert u\rVert_{L_{t,x}^{15}(\mathbb{R}\times\mathbb{R}^3)}\leq C\exp(CE^C)$.

In particular, if (\ref{eq-sc-uniform}) holds with $I=I_{\textrm{max}}$ then $I_{\textrm{max}}=\mathbb{R}$, and the associated $L_{t,x}^{15}(I\times \mathbb{R}^3)$ bound holds.
\end{proposition}

A related statement (without quantitative dependence on the {\it a priori} bound) was obtained for NLS with a class of 
energy-supercritical nonlinearities in \cite{Murphy}, based on concentration compactness methods (the paper \cite{Murphy} also treats the 
energy-subcritical problem; see also the discussion there for an overview of the related preceeding literature).  We follow the approach 
of \cite{Tao}, which is in turn related to the induction on energy ideas of Bourgain \cite{B}.  As we remarked above, the restriction to 
three spatial dimensions (and algebraic nonlinearity) allows us to focus on the simplest form of the argument.  

The quantitative bounds derived in Proposition \ref{prop-uniform} are of substantial independent interest.  As a second application of our techniques,
we prove the following corollary, which shows that slowly-growing control of the critical norm $\dot{H}_x^{s_c}$ by the scale-invariant Strichartz norm
$L_{t,x}^{15}(I\times\mathbb{R}^3)$ implies that a solution can be extended globally (and scatters at $\pm\infty$).

\begin{corollary}
\label{cor1}
Let $C$ be the constant in Proposition $\ref{prop-uniform}$, and let $g:[0,\infty)\rightarrow [0,\infty)$ be defined by $$g(t)=[C^{-1}\log(\log^{1/2}(t) )]^{1/C}.$$ Suppose that $u$ is a radially symmetric solution to (\ref{eq1}) with maximal interval of existence $I_{\textrm{max}}\subset\mathbb{R}$ such that for every interval $I\subset\subset I_{\textrm{max}}$
\begin{align*}
\lVert u\rVert_{L_t^\infty(I;\dot{H}_x^{s_c}(\mathbb{R}^3))}\leq g(\lVert u\rVert_{L_{t,x}^{15}(I\times\mathbb{R}^3)}).
\end{align*}
Then there exists $M_0>0$ so that $\lVert u\rVert_{L_{t,x}^{15}(I)}\leq M_0$ for all $I\subset I_{\textrm{max}}$.  In particular, $I_{\textrm{max}}=\mathbb{R}$.
\end{corollary}

\subsection*{Outline of the paper}

We briefly summarize the structure of this paper.  In Section $2$ we recall some preliminaries, including the standard homogeneous and inhomogeneous
Strichartz estimates for the linear Schr\"odinger equation.  In Section $3$ we prove Proposition \ref{prop-uniform}.  Then, in Section $4$ and $5$ we use
a standard continuity argument to give the proofs of Theorem $\ref{thm1}$ and Corollary $\ref{cor1}$.

\subsection*{Acknowledgements}

The author is very grateful to Terry Tao for many valuable conversations throughout the development of this work.

\section{Preliminaries and Strichartz estimates}

In this section, we fix our choice of function spaces and recall the corresponding Strichartz estimates which will be used 
in the rest of this paper.  We say that a pair $(q,r)$ with $2\leq q<\infty$ and $2\leq r<\infty$ is (Schr\"odinger) admissible on $\mathbb{R}^3$ 
if $\frac{2}{q}+\frac{3}{r}=\frac{3}{2}$.  For any such pair, we have
\begin{align*}
\lVert e^{it\Delta}u_0\rVert_{L_t^qL_x^r}\lesssim \lVert u_0\rVert_{L_x^2(\mathbb{R}^3)}
\end{align*}
and
\begin{align*}
\bigg\lVert \int_0^t e^{i(t-t')\Delta}F(t')dt'\bigg\rVert_{L_t^qL_x^r}\lesssim \lVert F\rVert_{L_t^{q_2'}L_x^{r_2'}}
\end{align*}
for every admissible pair $(q_2,r_2)$, where $q_2'$ and $r_2'$ denote conjugate exponents.

Let $I\subset\mathbb{R}$ be an open time interval, and suppose that $u\in C_t(I_0;\dot{H}_x^{s_c}(\mathbb{R}^3))\cap L_{t,x}^{15}(I_0\times\mathbb{R}^3)$ for all compact $I_0\subset I$.  For $I_0\subset I$, define
\begin{align*}
\lVert u\rVert_{S(I_0)}&=\lVert u\rVert_{L_{t,x}^{15}(I_0\times\mathbb{R}^3)}\\
\lVert u\rVert_{W(I_0)}&=\max\{\lVert |\nabla|^{s_c}u\rVert_{L_{t,x}^{10/3}(I_0\times\mathbb{R}^3)},\lVert |\nabla|^{s_c}u\rVert_{L_t^{15}(I_0;L_x^{\frac{90}{41}}(\mathbb{R}^3))}\}\\
\lVert u\rVert_{N(I_0)}&=\lVert |\nabla|^{s_c}u\rVert_{L_{t,x}^{10/7}(I_0\times\mathbb{R}^3)}.
\end{align*}
The norms $\lVert\cdot\rVert_{S(I_0)}$ and $\lVert\cdot\rVert_{W(I_0)}$ are chosen so that they can be controlled by appropriate Strichartz estimates, with the norm $\lVert\cdot\rVert_{N(I_0)}$ chosen by duality as an admissible choice of norm on the right-hand side of the inhomogeneous Strichartz bound.

In particular, one has (also using the Sobolev embedding)
\begin{align}
\lVert e^{it\Delta}u_0\rVert_{L_{t,x}^{15}}\lesssim \lVert |\nabla|^{s_c}e^{it\Delta}u_0\rVert_{L_t^{15}L_x^{\frac{90}{41}}}\lesssim \lVert u_0\rVert_{\dot{H}_x^{s_c}}.\label{str-linear}
\end{align}

\section{Proof of Proposition \ref{prop-uniform}}

In this section we prove Proposition \ref{prop-uniform}.  We begin by recalling some estimates for the localized $L^2$ norm $M(u;x_0,R)$, where $u$ solves either (\ref{eq1}) or the linear Schr\"odinger equation (we follow the treatment of \cite{Tao}).  These are based on the formal identity $\partial_t|u(t)|^2=-2\divop(\iIm(\overline{u(t)}(\nabla u)(t)))$ satisfied by sufficiently smooth solutions of either equation.

Let $\chi\in C^\infty_c(\mathbb{R}^3)$ be such that $\chi\equiv 1$ on $B(0;1/2)$ and $\supp\chi\subset B(0;1)$.  For each $x\in\mathbb{R}^3$, $R>0$, set
\begin{align*}
M(u;x_0,R)=\bigg(\int \Big|\chi\Big(\frac{x-x_0}{R}\Big)u(x)\Big|^2dx\bigg)^{1/2}.
\end{align*}
Then, using the identity mentioned above, one obtains
\begin{align*}
\partial_t[M(u(t);x_0,R)^2]=\frac{4}{R}\int\chi(\frac{x-x_0}{R})(\nabla \chi)(\frac{x-x_0}{R})\cdot\iIm(\overline{u}\nabla u)dx
\end{align*}
and thus
\begin{align*}
|\partial_t[M(u(t);x_0,R)^2]|\lesssim R^{-5/6}M(u(t);x_0,R)\lVert \nabla u\rVert_{L_x^{9/4}}
\end{align*}
so that
\begin{align*}
\partial_tM(u(t);x_0,R)\lesssim R^{-5/6}\lVert u\rVert_{L_t^\infty(I;\dot{H}_x^{s_c}(\mathbb{R}^3))}.
\end{align*}

Moreover, one also has
\begin{align}
|M(u(t);x_0,R)|\leq \lVert u\rVert_{L_x^9(\mathbb{R}^3)}\lVert \chi(x/R)\rVert_{L_x^{18/7}}\lesssim R^{7/6}\lVert u\rVert_{L_t^\infty(I;\dot{H}_x^{s_c}(\mathbb{R}^3))}\label{mass-bound}
\end{align}

We also recall a spatially localized form of the Morawetz estimate in our energy-supercritical setting.  Estimates of this type were originally obtained by Bourgain \cite{B} (see also the treatment in \cite{Tao} and related earlier work of Lin and Strauss).  In the energy-supercritical setting, we refer to the treatment of Murphy in \cite{Murphy}, which covers the case we need.  In the interest of completeness, we sketch the argument in our setting (following the presentation of \cite{B}).

\begin{proposition}[Spatially localized Morawetz estimate]
\label{prop-morawetz}
Suppose that $u$ solves (\ref{eq1}).  We then have
\begin{align*}
\int_I\int_{|x|<C|I|^{1/2}}\frac{|u(t,x)|^{8}}{|x|}dxdt&\lesssim (C|I|)^{2/3}\lVert u\rVert_{L_t^\infty(I;\dot{H}_x^{s_c})}
\end{align*}
for every time interval $I\subset\mathbb{R}$.
\end{proposition}

\begin{proof}
We argue as in \cite[Lemma 2.1]{B}.  Setting $r=|x|$ and $v=\iRe u$, $w=\iIm u$, we multiply (NLS) by $$\overline{\left(u_r+\tfrac{1}{r}u\right)}$$ and take the real part of both sides to obtain
\begin{align}
\partial_t X+\divop Y+Z=0\label{mor-eq-1}
\end{align}
with
\begin{align*}
X(t,x)&=-w(t,x)\left(\partial_r v(t,x)+\tfrac{1}{r}v(t,x)\right),\\
Y(t,x)&=\frac{x}{r}(\partial_t v)(t,x)w(t,x)-(\nabla v)(t,x)\left(\partial_r v(t,x)+\tfrac{1}{r}v(t,x)\right)\\
&\hspace{0.4in}-(\nabla w(t,x)\left(\partial_r w(t,x)+\tfrac{1}{r}w(t,x)\right)+\frac{x}{2r}|\nabla u(t,x)|^2\\
&\hspace{0.4in}+\frac{x}{8r}|u(t,x)|^{8}-\frac{x}{2r^3}|u(t,x)|^2\\
Z(t,x)&=\tfrac{1}{r}(|\nabla u|^2-|\partial_r u|^2)+\frac{2}{3r}|u|^{8}
\end{align*}

Now, consider $\phi:\mathbb{R}^3\rightarrow\mathbb{R}$ radial such that $\phi=1$ on $\{x:|x|<\delta\}$, $\supp \phi\subset \{x:|x|<2\delta\}$ and
\begin{align*}
|D^j\phi|\lesssim \delta^{-j},\quad j\geq 1.
\end{align*}
Multiplying ($\ref{mor-eq-1}$) by $\phi$ and integrating, it suffices to bound the quantities
\begin{align*}
\int |X|\phi dx\quad\textrm{and}\quad\left|\int Y\cdot \nabla\phi dx\right|
\end{align*}
(see the proof of \cite[Lemma 2.1]{B} for more details).  To fix ideas, we give the first estimate:
\begin{align*}
\int_{\mathbb{R}^d} |X|\phi dx&\leq \int_{\mathbb{R}^d} |u|\left(|\nabla u|+\tfrac{1}{r}|u|\right)\phi(x)dx\\
&\leq \lVert |\nabla u(t)|\phi^{1/2}\rVert_{L_x^2}+\lVert \tfrac{1}{r}|u(t)|\phi^{1/2}\rVert_{L_x^2}\lVert u(t)\phi^{1/2}\rVert_{L_x^2}\\
&\leq \bigg(\lVert \nabla u(t)\rVert_{L_x^{9/4}}\lVert \phi\rVert_{L_x^{9}}^{1/2}\\
&\hspace{0.6in}+\lVert \tfrac{1}{r}u(t)\rVert_{L_x^{9/4}}\lVert \phi\rVert_{L_x^{9}}^{1/2}\bigg)\lVert u(t)\rVert_{L_x^{9}}\lVert \phi\rVert_{L_x^{\frac{9}{7}}}^{1/2}\\
&\leq \delta^{2-\frac{2}{3}}\lVert u\rVert_{L_t^\infty\dot{H}_x^{s_c}}^2
\end{align*}

The other estimates follow similarly.
\end{proof}

\subsection{Preliminary construction}

Let $u$ be a radial solution to (\ref{eq1}) on a time interval $I=[t_-,t_+]$.  Without loss of generality, we may assume $E:=E_{s_c}\geq c_0$ for some $c_0>0$ (since $E$ sufficiently small implies the desired bound by the local theory of the previous subsection).  Fix parameters $1\leq C_0\leq C_1\leq C_2$ to be determined later in the argument, and set
\begin{align*}
\eta=\frac{1}{C_2}(1+E)^{-C_2}.
\end{align*}

Partition $I$ into consecutive disjoint intervals $\{I_1,I_2,\cdots,I_J\}$ with each $I_j=[t_j,t_{j+1}]$ so that
\begin{align*}
\int_{t_j}^{t_{j+1}}\int |u(t,x)|^{15}dxdt\in [\eta,2\eta].
\end{align*}

\begin{remark}[Absorbing $E$ into expressions involving $\eta$]
\label{rem-absorb}
Let $p>0$ be given.  For each $\epsilon>0$ there exists $\overline{C_2}=\overline{C_2}(p)$ so that for all $C_2>\overline{C_2}$ we have $E\eta^p<\epsilon$.  Indeed, writing $E\eta^p=(E/(1+E)^{C_2p})C_2^{-p}<(1+E)^{1-C_2p}C_2^{-p}$, the condition $1-C_2p<0$ leads to $E\eta^p\leq C_2^{-p}$, and thus the additional condition $C_2>\epsilon^{-1/p}$ implies the desired bound.

We also have that for each $C>0$ and $p>0$ there exists $C'>0$ so that $\eta^CE^{-p}\geq \eta^{C'}$.  Indeed, writing $\eta^CE^{-p}\geq \eta^C(1+E)^{-p}=\eta^C(C_2\eta)^{p/C_2}$ and recalling $C_2\geq 1$, one obtains the desired bound with $C'=C+(p/C_2)$.

Similarly, for fixed $p>0$, $C>0$, and $\epsilon>0$, we can choose $C_2$ sufficiently large to ensure $E^p\eta^{-C}\leq \eta^{-C-\epsilon}$.  This follows by writing $E^p\eta^{-C}\leq (1+E)^p\eta^{-C}=C_2^{-p/C_2}\eta^{-C-(p/C_2)}$, and noting that $C_2\geq 1$ therefore implies $E^p\eta^{-C}\leq \eta^{-C-(p/C_2)}$.  The condition $C_2>p/\epsilon$ now implies (in view of $0<\eta<1$) $\eta^{C+\epsilon}<\eta^{C+(p/C_2)}$, and thus one has $E^p\eta^{-C}\leq \eta^{-C-\epsilon}$ as desired.
\end{remark}

\begin{lemma}
\label{lem1}
There exists $C>0$ such that for each $I_0\subset I$, if
\begin{align*}
\int_{I_0}\int |u(t,x)|^{15}dxdt\leq 2\eta
\end{align*}
then
\begin{align*}
\lVert u\rVert_{W(I_j)}\leq CE.
\end{align*}
\end{lemma}

\begin{proof}
\begin{align*}
\lVert u\rVert_{W(I_j)}&\lesssim \lVert u(t_j)\rVert_{\dot{H}_x^{s_c}}+\lVert |u|^6u\rVert_{N(I_j)}\\
&\lesssim E+\lVert |\nabla|^{s_c}u\rVert_{L_{t,x}^{10/3}(I_j)}\lVert u\rVert_{S(I_j)}^6\\
&\lesssim E+\lVert u\rVert_{W(I_j)}\eta^{6/15}.
\end{align*}
Choosing $\eta$ sufficiently small now ensures $\lVert u\rVert_{W(I_j)}\lesssim E$ as desired.
\end{proof}

\begin{corollary}
There exists $C>0$ and $\overline{C_2}>0$ so that if $C_2\geq \overline{C_2}$ then for each $1\leq j\leq J$, then one has $\lVert u\rVert_{W(I_j)}\leq CE$.
\end{corollary}

The next lemma is not used in our subsequent argument (it is the analogue of Lemma 3.2 in \cite{Tao}, and is not needed since we are working only in dimension $3$), but we record it here as a preliminary version of a similar argument used in the proof of Proposition $\ref{prop37}$ below to estimate a suitable norm of the function $v$.

\begin{lemma}
There exists $c>0$ such that for each $I_0=[t_1,t_2]\subset I$ satisfying
\begin{align*}
\int_{t_1}^{t_2}\int |u(t,x)|^{15}dxdt\in [\frac{\eta}{2},2\eta],
\end{align*}
we have
\begin{align*}
\int_{t_1}^{t_2}\int |e^{i(t-t_i)\Delta}u(t_i)|^{15}dxdt\geq c\eta
\end{align*}
for $i\in \{1,2\}$.
\end{lemma}

\begin{proof}
Using the Sobolev embedding
\begin{align*}
\lVert u-e^{i(t-t_i)\Delta}u(t_i)\rVert_{L_{t,x}^{15}(I_0\times\mathbb{R}^3)}&\lesssim \lVert |\nabla|^{s_c}(u-e^{i(t-t_i)\Delta}u(t_i))\rVert_{L_t^{15}(I_0;L_x^{\frac{90}{41}}(\mathbb{R}^3))}
\end{align*}
one has
\begin{align*}
\lVert u-e^{i(t-t_i)\Delta}u(t_i)\rVert_{L_{t,x}^{15}(I_0\times\mathbb{R}^3)}&\lesssim \lVert u-e^{i(t-t_i)\Delta}u(t_i)\rVert_{W(I_0)}\\
&\lesssim \lVert |u|^6u\rVert_{N(I_0)}\\
&\lesssim \lVert |\nabla|^{s_c}u\rVert_{W(I_0)}\lVert u\rVert_{S(I_0)}^6
\end{align*}
so that by Lemma \ref{lem1} one gets the bound
\begin{align*}
\lVert u-e^{i(t-t_i)\Delta}u(t_i)\rVert_{L_{t,x}^{15}(I_0\times\mathbb{R}^3)}\leq CE\eta^{6/15}.
\end{align*}

Using this, we now get
\begin{align*}
(\eta/2)^{1/15}&\leq \lVert u\rVert_{L_{t,x}^{15}(I_0\times\mathbb{R}^3)}\\
&\leq \lVert u-e^{i(t-t_i)\Delta}u(t_i)\rVert_{L_{t,x}^{15}(I_0\times\mathbb{R}^3)}+\lVert e^{i(t-t_i)\Delta}u(t_i)\rVert_{L_{t,x}^{15}(I_0\times\mathbb{R}^3)}\\
&\leq CE\eta^{6/15}+\lVert e^{i(t-t_i)\Delta}u(t_i)\rVert_{L_{t,x}^{15}(I_0\times\mathbb{R}^3)}
\end{align*}
and thus choosing $C_2$ large enough ensures 
\begin{align*}
\lVert e^{i(t-t_i)\Delta}u(t_i)\rVert_{L_{t,x}^{15}(I_0\times\mathbb{R}^3)}\geq c\eta^{1/15}
\end{align*}
as desired.
\end{proof}

\subsection{Classification of intervals}

Set $u_-:=e^{i(t-t_-)\Delta}u(t_-)$ and $u_+=e^{i(t-t_+)\Delta}u(t_+)$.  Let $B$ denote the set of $j\in \{1,\cdots,J\}$ such that
\begin{align*}
\max\{\int_{t_j}^{t_{j+1}}\int |u_-(t)|^{15}dxdt,\int_{t_j}^{t_{j+1}}\int |u_+(t)|^{15}dxdt\}>\eta^{C_1}
\end{align*}
and set $G=\{1,\cdots,J\}\setminus B$.  In the language of \cite{Tao}, the intervals $I_j$, $j\in G$ are {\it unexceptional} intervals, and the intervals $I_j$, $j\in B$ are {\it exceptional} intervals.

\begin{remark}
\label{rem-count}
One immediately has a bound on $\#B$.  Indeed, first note that 
\begin{align*}
\sum_{j\in B} \int_{I_j}\int |u_+(t)|^{15}+|u_-(t)|^{15}dxdt> (\#B)\eta^{C_1}.
\end{align*}
On the other hand, by the Strichartz estimate (\ref{str-linear}), we have
\begin{align*}
\sum_{j\in B} \int_{I_j}\int |u_+(t)|^{15}+|u_-(t)|^{15}dxdt&\leq \int_{I}\int |u_+(t)|^{15}dxdt+\int_I\int |u_-(t)|^{15}dxdt\\
&\lesssim \lVert u(t_+)\rVert_{\dot{H}_x^{s_c}}^{15}+\lVert u(t_-)\rVert_{\dot{H}_x^{s_c}}^{15}\lesssim E^{15},
\end{align*}
so that $\#B\leq CE^{15}/\eta^{C_1}$.
\end{remark}

\subsection{Concentration bound}

\begin{proposition}
\label{prop37}
There exist $c,C>0$ such that for each $j\in G$ there exists $x_j\in\mathbb{R}^3$ such that
\begin{align*}
M(u(t);x_j,C\eta^{-C}|I_j|^{1/2})\geq c\eta^{C}E^{-3/2}|I_j|^{7/12}.
\end{align*}
\end{proposition}

\begin{proof}
Let $j\in G$ be given.  For simplicity of notation, set $J=I_j$ and recall that 
\begin{align*}
\int_{J}\int |u(t,x)|^{15}dxdt>\eta.
\end{align*}

Then, setting $J=I_j$, $t_*=(t_j+t_{j+1})/2$, and writing $J=J_1\cup J_2$ with $J_1=[t_j,t_*]$, $J_2=[t_*,t_{j+1}]$, we can choose $k\in \{1,2\}$ so that
\begin{align*}
\int_{J_k}\int |u(t,x)|^{15}dxdt>\eta/2.
\end{align*}

Moreover, $j\in G$ implies
\begin{align*}
\int_{J}\int |u_{k}(t,x)|^{15}dxdt\leq \eta^{C_1}.
\end{align*}
where $u_1=u_+$ and $u_2=u_-$.

Now, set $M_1=[t_{j+1},t_+]$ and $M_2=[t_-,t_j]$, and define
\begin{align*}
v(t):=\int_{M_k} e^{i(t-t')\Delta}[|u(t')|^6u(t')]dt'.
\end{align*}

We begin by recording some instances of the Duhamel formula which will be useful in our estimates: for each $t\in J_k$, one has
\begin{align}
u(t)=e^{i(t-t_+)\Delta}u(t_+)-iv(t)-i\int_{t}^{t_{j+1}} e^{i(t-t')\Delta}[|u(t')|^6u(t')]dt'\label{eq-k1}
\end{align}
if $k=1$, and 
\begin{align}
u(t)=e^{i(t-t_-)\Delta}u(t_-)-iv(t)-i\int_{t_j}^{t} e^{i(t-t')\Delta}[|u(t')|^6u(t')]dt'\label{eq-k2}
\end{align}
if $k=2$.  

The Duhamel formulas (\ref{eq-k1}) and (\ref{eq-k2}) combined with the Strichartz estimates yield an upper bound on the $L_t^\infty\dot{H}_x^{s_c}$ norm of $v$ (on the interval $J_k$)
\begin{align}
\nonumber \lVert v\rVert_{L_t^\infty(J_k;\dot{H}_x^{s_c}(\mathbb{R}^3))}&\lesssim \lVert u\rVert_{L_t^\infty(J_k;\dot{H}_x^{s_c}(\mathbb{R}^3)}+\lVert u(t_+)\rVert_{\dot{H}_x^{s_c}}+\lVert |u|^6u\rVert_{N(J_k)}\\
\nonumber &\lesssim 2E+\lVert |\nabla|^{s_c}u\rVert_{W(J)}\lVert u\rVert_{S(J)}^6\\
&\lesssim 2E+E\eta^{6/15}\lesssim E.\label{eq-c}
\end{align}

We now show the bound 
\begin{align}
\lVert v\rVert_{L_{t,x}^{15}(J_k\times\mathbb{R}^3)}\geq c\eta^{1/15}.\label{claim1}
\end{align}
For this, we note that when $k=1$, (\ref{eq-k1}) leads to
\begin{align*}
(\eta/2)^{1/15}&<\lVert u\rVert_{L_{t,x}^{15}(J_1\times\mathbb{R}^3)}
\leq \eta^{C_1/15}+\lVert u(t)-e^{i(t-t_{j+1})\Delta}u(t_{j+1})\rVert_{S(J_1)}+\lVert v\rVert_{S(J_1)},
\end{align*}
while when $k=2$, (\ref{eq-k2}) gives 
\begin{align*}
(\eta/2)^{1/15}&<\lVert u\rVert_{L_{t,x}^{15}(J_k\times\mathbb{R}^3)}
\leq \eta^{C_1/15}+\lVert u(t)-e^{i(t-t_j)\Delta}u(t_j)\rVert_{S(J_2)}+\lVert v\rVert_{S(J_2)}.
\end{align*}
On the other hand, the Sobolev embedding, Strichartz estimates, and fractional product rules imply
\begin{align*}
\lVert u(t)-e^{i(t-t_j)\Delta}u(t_j)\rVert_{L_{t,x}^{15}(J_k\times\mathbb{R}^3)}&\lesssim \lVert u-e^{i(t-t_j)\Delta}u(t_j)\rVert_{W(J_k)}\\
&\lesssim \lVert |u|^6u\rVert_{N(J_k)}\\
&\lesssim \lVert |\nabla|^{s_c}u\rVert_{W(J)}\lVert u\rVert_{S(J)}^6\\
&\lesssim E\eta^{6/15}.
\end{align*}
in both of the cases $k=1$ and $k=2$.  Combining this with the above estimates from below yields the claim.

Now, let $\chi_0\in C_c^\infty(\mathbb{R}^3)$ be such that $\chi_0(x)\geq 0$ for all $x\in\mathbb{R}^3$, $\supp\chi_0\subset B(0;1)$, and $\int \chi_0(x)dx=1$, and define
\begin{align*}
v_{\textrm{av}}(t,x):=\int \chi(y)v(t,x+ry)dy,
\end{align*}
where $r>0$ is a fixed parameter to be determined later in the argument.

We now claim that there exist $c,C>0$ so that
\begin{align}
\lVert v_{\textrm{av}}-v\rVert_{L_{t,x}^{15}(J_k\times \mathbb{R}^3))}\leq CE^{7}|J_k|^{-3/10}r^c.\label{eq2}
\end{align}

For this, for each $h\in\mathbb{R}^3$, let $T_h:f\mapsto T_hf$ be the operator defined by $(T_hf)(x)=f(x+h)$ for $f\in \mathcal{S}(\mathbb{R}^3)$ and $x\in \mathbb{R}^3$.  Then, since the Fourier multiplier operator $e^{is\Delta}$ commutes with translations, 
\begin{align*}
&\lVert v-T_hv\rVert_{L_{t}^\infty(J_k;L_x^{15}(\mathbb{R}^3))}\\
&\hspace{0.2in}=\bigg\lVert \int_{M_k} e^{i(t-t')\Delta}[|u(t')|^6u(t')-|T_hu(t')|^6T_hu(t')]dt'\bigg\rVert_{L_{t}^{\infty}(J_k;L_x^{15}(\mathbb{R}^3))}\\
&\hspace{0.2in}\leq \bigg\lVert\int_{M_k} \lVert e^{i(t-t')\Delta}[|u(t')|^6u(t')-|T_hu(t')|^6T_hu(t')]\rVert_{L_{x}^{15}(\mathbb{R}^3)}dt'\bigg\rVert_{L_t^\infty(J_k)}\\
&\hspace{0.2in}\lesssim \bigg\lVert \int_{M_k} |t-t'|^{-13/10}\lVert |u(t')|^6u(t')-|T_hu(t')|^6T_hu(t')\rVert_{L_x^{15/14}(\mathbb{R}^3)}dt'\bigg\rVert_{L_t^\infty(J_k)}\\
&\hspace{0.2in}\lesssim \bigg\lVert \int_{M_k} |t-t'|^{-13/10}\lVert u(t')-T_hu(t')\rVert_{L_x^{15/4}(\mathbb{R}^3)}\lVert u\rVert_{L_t^\infty(I;L_x^9(\mathbb{R}^3))}^6dt'\bigg\rVert_{L_t^\infty(J_k)}.
\end{align*}

Now, using the difference quotient bound
\begin{align*}
\lVert u-T_hu\rVert_{L_x^{9/4}}&\lesssim \lVert \nabla u\rVert_{L_x^{9/4}}|h|,
\end{align*}
interpolation and the Sobolev embedding give
\begin{align*}
\lVert u(t')-T_hu(t')\rVert_{L_x^{15/4}}&\leq 2\lVert u(t')\rVert_{L_x^9}^{8/15}\lVert \nabla u(t')\rVert_{L_x^{9/4}}^{7/15}|h|^{7/15}\\
&\lesssim \lVert |\nabla|^{s_c}u\rVert_{L_t^\infty L_x^2}|h|^{7/15} 
\end{align*}
for all $t'\in M_k$.  Recalling once more the Sobolev embedding $\lVert u\rVert_{L_x^9}\lesssim\lVert |\nabla|^{s_c}u\rVert_{L_x^2}$, we therefore get
\begin{align*}
\lVert v-T_hv\rVert_{L_{t}^\infty(J_k;L_x^{15}(\mathbb{R}^3))}\lesssim E^7|h|^{7/15}\sup_{t\in J_k}\int_{M_k} \frac{1}{|t-t'|^{13/10}}dt'\lesssim E^7|h|^{7/15}|J_k|^{-3/10},
\end{align*}
where in obtaining the last bound we have used that for $t\in J_k$, $\int_{M_k} |t-t'|^{-13/10}dt'\lesssim \dist(t,M_k)^{-3/10}\leq (|J_k|/2)^{-3/10}$.

To obtain (\ref{eq2}), we now write
\begin{align*}
\lVert v_{\textrm{av}}-v\rVert_{L_{t,x}^{15}(J_k\times\mathbb{R}^3)}&=\lVert \int \chi(y)(v(t,x+ry)-v(t,x))dx\rVert_{L_{t,x}^{15}(J_k\times\mathbb{R}^3)}\\
&\leq |J_k|^{1/15}\int \chi(y)\lVert T_{ry}v-v\rVert_{L_{t}^\infty(J_k;L_{x}^{15}(\mathbb{R}^3))}dy\\
&\lesssim E^7r^{7/15}|J_k|^{-7/30}\int \chi(y)|y|^{7/15}dy.
\end{align*}

Combining (\ref{claim1}) with (\ref{eq2}), we get
\begin{align*}
c\eta^{1/15}\leq \lVert v\rVert_{L_{t,x}^{15}(J_k\times \mathbb{R}^3)}&\leq \lVert v-v_{\textrm{av}}\rVert_{L_{t,x}^{15}(J_k\times\mathbb{R}^3)}+\lVert v_{\textrm{av}}\rVert_{L_{t,x}^{15}(J_k\times\mathbb{R}^3)}\\
&\leq CE^7r^{7/15}|J_k|^{-7/30}+C\lVert v_{\textrm{av}}\rVert_{L_{t,x}^{15}(J_k\times\mathbb{R}^3)}
\end{align*}
so that choosing $r$ sufficiently small to ensure $CE^7|J_k|^{-7/30}r^{7/15}\leq c\eta^{1/15}/2$ 
gives
\begin{align}
\lVert v_{\textrm{av}}\rVert_{L_{t,x}^{15}(J_k\times\mathbb{R}^3)}\geq c\eta^{1/15}.\label{eq-a}
\end{align}

In what follows, we fix
\begin{align*}
r:=\eta^{C}|J_k|^{1/2}
\end{align*}
with $C$ sufficiently large to ensure that (\ref{eq-a}) holds (in particular, note that the condition (\ref{eq-a}) is equivalent to $\eta^{(7C-1)/15}E^7\leq c$ for suitable $c>0$, so that it suffices to choose $C>2/7$ and $C_2$ sufficiently large to ensure $E\eta^{1/105}\leq c^{1/7}$).

To proceed, we now obtain a complementary estimate from above on $\lVert v_{\textrm{av}}\rVert_{L_{t,x}^{15}}$.  In particular, by interpolation,
\begin{align}
\nonumber \lVert v_{\textrm{av}}\rVert_{L_{t,x}^{15}(J_k\times\mathbb{R}^3)}&\leq \lVert v_{\textrm{av}}\rVert_{L_{t,x}^{9}(J_k\times\mathbb{R}^3)}^{3/5}\lVert v_{\textrm{av}}\rVert_{L_{t,x}^\infty(J_k\times\mathbb{R}^3)}^{2/5}\\
\nonumber &\lesssim |J_k|^{1/15}\lVert v_{\textrm{av}}\rVert_{L_t^\infty(J_k;L_x^{9}(\mathbb{R}^3))}^{3/5}\lVert v_{\textrm{av}}\rVert_{L_{t,x}^\infty(J_k\times\mathbb{R}^3)}^{2/5}\\
&\lesssim |J_k|^{1/15}\lVert v\rVert_{L_t^\infty(J_k;L_x^{9}(\mathbb{R}^3))}^{3/5}\lVert v_{\textrm{av}}\rVert_{L_{t,x}^\infty(J_k\times\mathbb{R}^3)}^{2/5},\label{eq-b}
\end{align}
where to obtain the last inequality we have defined $\chi_r(x)=r^{-3}\chi(x/r)$ and used (via Young's inequality)
\begin{align*}
\lVert v_{\textrm{av}}\rVert_{L_t^\infty(J_k;L_x^9(\mathbb{R}^3))}&=\esssup_{t\in J_k}\lVert \chi_r\star v\rVert_{L_x^9(\mathbb{R}^3)}\\
&\lesssim \esssup_{t\in J_k} \lVert \chi_r\rVert_{L_x^1(\mathbb{R}^3)}\lVert v\rVert_{L_x^9(\mathbb{R}^3)}
\end{align*}
and noted that $\lVert \chi_r\rVert_{L_x^1}=\lVert \chi\rVert_{L_x^1}=1$.

In view of (\ref{eq-c}) (combined with the Sobolev embedding $\dot{H}_x^{s_c}(\mathbb{R}^3)\hookrightarrow L_x^9(\mathbb{R}^3)$), the estimates (\ref{eq-a}) and (\ref{eq-b}) lead to
\begin{align*}
\lVert v_{\textrm{av}}\rVert_{L_{t,x}^{\infty}}&\geq c\eta^{1/2}|J_k|^{-1/6}E^{-3/2}
\end{align*}
so that we may find $s_*\in J_k$ and $x_*\in \mathbb{R}^3$ with
\begin{align*}
c\eta^{1/2}|J_k|^{-1/6}E^{-3/2}\leq v_{\textrm{av}}(s_*,x_*)&\lesssim r^{-3/2}M(v(s_*);x_0,r),
\end{align*}
where to obtain the last inequality we have recalled the equality $v_{\textrm{av}}(s_*,x_*)=\frac{1}{r^3}\int_{\mathbb{R}^3} \chi(\frac{y-x_*}{r})v(s_*,y)dy$ (indeed, this is just the definition of $v_{\textrm{av}}$).

To conclude the argument, we now transfer this lower bound from $v$ back to $u$.  We consider the cases $k=1$ and $k=2$ individually.  Suppose first that we are in the case $k=1$ and fix $\lambda>0$ to be determined later.  Then for all $t\in J$,
\begin{align*}
M(u(t);x_*,\lambda)&=M(u(t_{j+1});x_*,\lambda)-\int_{t}^{t_{j+1}} \partial_s M(u(s);x_*,\lambda)ds\\
&\geq M(u(t_{j+1});x_*,\lambda)-CE|J_1|\lambda^{-5/6}
\end{align*}
so that since (\ref{eq-k1}) implies
\begin{align*}
u(t_{j+1})=e^{i(t_{j+1}-t_+)\Delta}u(t_+)-iv(t_{j+1})=u_+(t_{j+1})-iv(t_{j+1}).
\end{align*}
By Minkowski's inequality (with the measure $\chi(\frac{x-x_0}{\lambda})^2dx$) this gives
\begin{align*}
M(u(t);x_*,\lambda)&\geq M(v(t_{j+1});x_*,\lambda)-M(u_+(t_{j+1});x_*,\lambda)-CE|J_1|\lambda^{-5/6}.
\end{align*}
Since $u_+$ and $v$ solve the linear Schr\"odinger equation, we get
\begin{align*}
M(v(t_{j+1});x_*,\lambda)&=M(v(s_*);x_*,\lambda)+\int_{s_*}^{t_{j+1}} \partial_sM(v(s);x_*,\lambda)ds\\
&\geq c\eta^{1/2}|J_1|^{-1/6}E^{-3/2}r^{3/2}-C|J_1|\lVert v\rVert_{L_t^\infty(J_1;\dot{H}_x^{s_c}(\mathbb{R}^3))}\lambda^{-5/6}\\
&\geq c\eta^{1/2}|J_1|^{-1/6}E^{-3/2}r^{3/2}-CE|J_1|\lambda^{-5/6}
\end{align*}
provided $\lambda\geq r$, and, noting that $\lVert u_+\rVert_{S(J_1)}\leq \eta^{C_1/15}$ implies that for some $\sigma\in J_1$ we have $\lVert u_+(\sigma)\rVert_{L_x^{15}(\mathbb{R}^3)}\leq (2\eta^{C_1}|J_1|^{-1})^{1/15}$, 
\begin{align*}
M(u_+(t_{j+1});x_*,\lambda)&=M(u_+(\sigma);x_*,\lambda)+\int_{\sigma}^{t_{j+1}} \partial_sM(u_+(s);x_*,\lambda)ds\\
&\leq C\eta^{C_1/15}|J_1|^{-1/15}\lambda^{13/10}+(C|J_1|\lambda^{-5/6})\lVert u_+\rVert_{L_t^\infty(J_1;\dot{H}_x^{s_c}(\mathbb{R}^3))}\\
&\lesssim C\eta^{C_1/15}|J_1|^{-1/15}\lambda^{13/10}+(C|J_1|\lambda^{-5/6})E
\end{align*}

Fix $C'>0$ to be determined momentarily.  Assembling these estimates, and choosing $\lambda=\eta^{-C'}|J_1|^{1/2}\geq r=\eta^{C}|J_1|^{1/2}$, one has
\begin{align*}
M(u(t);x_*,\lambda)&\geq c\eta^{1/2}|J_1|^{-1/6}E^{-3/2}r^{3/2}-C\eta^{C_1/15}|J_1|^{-1/15}\lambda^{13/10}-3CE|J_1|\lambda^{-5/6}\\
&\geq \Big(c\eta^{(1+3C)/2}E^{-3/2}-(C\eta^{\frac{C_1}{15}-\frac{13C'}{10}}+3CE\eta^{\frac{5C'}{6}})\Big)|J_1|^{7/12}.
\end{align*}

An identical argument applies in the case $k=2$.  Appropriate choice of $C'$ and $C_1$ now gives
\begin{align*}
M(u(t);x_*,\eta^{-C'}|J|^{1/2})\geq c\eta^C|J|^{7/12}.
\end{align*}
where we have recalled that $|J_k|=\frac{1}{2}|J|$ and used the second observation in Remark $\ref{rem-absorb}$.
\end{proof}

Since we are in the radial case, we immediately get a ``centered'' version of the concentration result.

\begin{corollary}
\label{cor-concentration}
There exist $c,C>0$ such that for each $j\in G$ 
\begin{align*}
M(u(t);0,C\eta^{-C}|I_j|^{1/2})\geq c\eta^{C}|I_j|^{7/12}.
\end{align*}
\end{corollary}

\subsection{Morawetz control}

Let $j\in G$ be given, and set $R=C\eta^{-C}|I_j|^{1/2}$.  The concentration bounds of the previous subsection, together with the estimate
\begin{align*}
M(u(t);0,R)&\leq \lVert u\rVert_{L_x^8(\{x:|x|\leq CR\})}\lVert \chi(\frac{x}{R})\rVert_{L_x^{8/3}}\\
&\lesssim \bigg(\int_{\{x\in \mathbb{R}^3:|x|\leq CR\}} \frac{|u(t)|^8}{R}dx\bigg)^{1/8}R^{10/8},
\end{align*}
now give
\begin{align*}
(c\eta^C|I_j|^{7/12})^8R^{-10}\leq \int_{|x|\leq C\eta^{-C}|I_j|^{1/2}} \frac{|u(t,x)|^8}{|x|}dx.
\end{align*}
This in turn leads to
\begin{align}
\int_{I_j}\int_{|x|\leq C\eta^{-C}|I_j|^{1/2}} \frac{|u(t,x)|^8}{|x|}dx\geq c\eta^{9C}|I_j|^{2/3}.\label{eq-mor1}
\end{align}

Combining this with the Morawetz estimate (Proposition \ref{prop-morawetz}), we obtain 
\begin{proposition}
\label{prop-morawetz-bd}
If $I_0\subset I$ is a union of consecutive intervals $I_j$, $j_-\leq j\leq j_+$, then there exists $j_*$ with $j_-\leq j_+$ so that 
\begin{align*}
|I_{j_*}|\geq c\eta^{3C_1/2}|I|.
\end{align*}
\end{proposition}

\begin{proof}
It suffices to show
\begin{align*}
\sup_{j_-\leq j\leq j_+} |I_j|\geq c\eta^{3C_1/2}|I|.
\end{align*}
For this, we recall (\ref{eq-mor1}) above and Remark $\ref{rem-count}$, from which one gets the bound
\begin{align*}
|I_0|=\sum_{j_-\leq j\leq j_+} |I_j|&\leq (\sup_j |I_j|)^{1/3}(\sum_j |I_j|^{2/3})\\
&\leq (\sup_j |I_j|)^{1/3}(\sum_{j\in G} |I_j|^{2/3}+\sum_{j\in B} |I_j|^{2/3})\\
&\leq (\sup_j |I_j|)^{1/3}\bigg(\sum_{j\in G} C\eta^{-9C}\int_{I_j}\int_{\mathbb{R}^3} \frac{|u|^8}{|x|}dxdt+(\#B)|I_0|^{2/3}\bigg)\\
&\leq (\sup_j |I_j|)^{1/3}(c\eta^{-9C}E+CE^{3}\eta^{-C_1})|I_0|^{2/3}\\
&\leq (\sup_j |I_j|)^{1/3}C\eta^{-C_1}|I_0|^{2/3}.
\end{align*}
This implies the desired inequality.
\end{proof}

\subsection{Recursive control over unexceptional intervals}

In this subsection, we reproduce an argument originally due to Bourgain \cite{B} which shows that there exists a time $t_*\in [t_-,t_+]$ at which the intervals $I_j$ concentrate.  We follow the presentation in \cite[Proposition 3.8]{Tao} (see also a related treatment in \cite{KVZ}).  Each step in the process is a consequence of the Morawetz-based control of the previous two subsections, and is expressed in the following pair of lemmas (stated separately for notational convenience).

In what follows, we set $J'=\#G$.  For a collection of sets $\mathcal{S}$, we also use the notation $\cup \mathcal{S}=\bigcup_{S\in\mathcal{S}} S$.

\begin{lemma}
\label{lem-1}
There exist $j_-,j_+\in \{1,2,\cdots,J\}$ so that
\begin{enumerate}
\item[(i)] $j_+-j_->c\eta^{C_1/2}J'$,
\item[(ii)] $j\in G$ for $j_-\leq j\leq j_+$, and
\item[(iii)] there exists $j_1\in\mathbb{N}$ with $j_-\leq j_1\leq j_+$ so that $$|I_{j_1}|\geq c\eta^{C}|\bigcup_{j_-\leq j\leq j_+} I_j|.$$
\end{enumerate}
\end{lemma}

\begin{proof}
Recall that $\#B\leq CE^{15}\eta^{-C_1}\leq C\eta^{-C_1/2}$, and partition $G=\{1,\cdots,J\}\setminus B$ into a collection $\{G_1,G_2,\cdots,G_m\}$ of nonempty sets of consecutive integers, with $m\leq 2C\eta^{-C_1/2}$.  Then there exists $i\in \{1,\cdots,m\}$ so that $G_i$ contains at least $(4C)^{-1}\eta^{C_1/2}J'$ elements (if this were not the case, then the total number of elements in $G$ would be at most $J'/2$).  Set $j_-=\min G_i$ and $j_+=\max G_i$ and observe that this choice satisfies (i) and (ii).

To establish (iii), we appeal to Proposition \ref{prop-morawetz-bd}, which gives the existence of $j_1$ satisfying the stated conditions.  This completes the proof of the lemma.
\end{proof}

If $j_+-j_-$ is larger than (a suitable multiple of) $\eta^{-C}$, this procedure can be iterated, this time removing $j_1$ and all intervals of comparable or longer length, rather than exceptional intervals.

\begin{lemma}
\label{lem2}
Let $j_0,j_1,j_2\in \{1,2,\cdots,J\}$ be given with $1\leq j_0<j_1<j_2\leq J$, and suppose $j_2-j_0>\tilde{C}\eta^{-C}$ and $|I_{j_1}|\geq \tilde{c}\eta^{C}|\cup \{I_j:j_0\leq j\leq j_2\}|$ with $\tilde{C}$ sufficiently large depending on $\tilde{c}$.  Then there exist $j_-$ and $j_+$ with $j_0\leq j_-<j_+\leq j_2$ so that
\begin{enumerate}
\item[(i)] $j_+-j_->c\eta^{C}(j_2-j_0)$,
\item[(ii)] $j_1\not\in [j_-,j_+]$ and $|I_j|\leq |I_{j_1}|/2$ for $j_-\leq j\leq j_+$, and
\item[(iii)] there exists $j_*\in \mathbb{N}$ with $j_-\leq j_*\leq j_+$ so that $$|I_{j_*}|\geq c\eta^{C}|\bigcup_{j_-\leq j\leq j_+} I_j|.$$
\end{enumerate}
\end{lemma}

\begin{proof}
The argument is similar to the proof of the previous lemma.  Define $A=\{j_0,j_0+1,\cdots,j_2\}\setminus (\{j_1\}\cup\{j:j_0\leq j\leq j_1$ and $|I_j|\geq |I_{j_1}|/2\})$, and partition $A$ into a collection of $m$ nonempty sets of consecutive integers.  Note that since the removed set has size bounded by $1+C\eta^{-C}\leq 2C\eta^{-C}$ (if the interval $I'=\cup\{I_j:j_0\leq j\leq j_2\}$ contained more than $2c^{-1}\eta^{-C}$ disjoint intervals of length at least $|I_{j_1}|/2\geq c\eta^{C}|I'|/2$, this would imply $|I'|\geq 2|I'|$, a contradiction), one can form such a partition with $$m\leq 2C\eta^{-C}.$$  One can now choose $c>0$ so that at least one of these sets contains more than $c\eta^{C}(j_2-j_0)$ elements (otherwise the total number of elements in $A$ is bounded by a multiple of $2Cc(j_2-j_0)$, while $\#A\geq (j_2-j_0)-4\tilde{c}^{-1}\eta^{-C}\geq (1-4(\tilde{C}\tilde{c})^{-1})(j_2-j_0)$, which is a contradiction for $c$ large).  Denote the set chosen this way by $A'$.  The claim now follows by choosing $j_-=\min A'$ and $j_+=\max A'$, and applying Proposition $\ref{prop-morawetz-bd}$ to obtain the existence of $j_*$ as in (iii).
\end{proof}

Applying Lemma $\ref{lem-1}$ and subsequently repeatedly applying Lemma $\ref{lem2}$ a finite number of times yields the following proposition (which is the analogue in our setting of Proposition 3.8 in \cite{Tao}).

\begin{proposition}
\label{prop-recursive}
There exist constants $c,C>0$ and values $K>c\eta^C\log(J')$ and $t_*\in [t_-,t_+]$ so that one can select $K$ distinct unexceptional intervals $I_{j_1},\cdots,I_{j_K}$, with 
\begin{align*}
|I_{j_1}|\geq 2|I_{j_2}|\geq\cdots\geq 2^{K-1}|I_{j_K}|,
\end{align*}
and, for $1\leq k\leq K$, $\dist(t_*,I_{j_k})\leq C\eta^{-C}|I_{j_k}|$.
\end{proposition}

\begin{proof}
First, apply Lemma $\ref{lem-1}$ to find $j_-$, $j_+$ and $j_1$ satisfying the conditions given by the statement of the lemma.  Then for all $t\in \tilde{I}_1:=\cup\{I_j:j_-\leq j\leq j_+\}$ we have $$\dist(t,I_{j_1})\leq |\tilde{I}_1|\leq c^{-1}\eta^{-C}|I_{j_1}|.$$  Let $\tilde{C}$ be a constant admissible for the statement of Lemma $\ref{lem2}$ where $\tilde{c}$ is chosen as $c$ in (iii) of Lemma $\ref{lem-1}$.  If $j_+-j_-\leq \tilde{C}\eta^{-C}$, we can set $K=1$ and choose any $t_*\in I_{j_1}$ to obtain the conclusion.

Alternatively, if $j_+-j_->\tilde{C}\eta^{-C}$ then we can apply Lemma $\ref{lem2}$ with $j_0=j_-$ and $j_2=j_+$ to find new values of $j_-$ and $j_+$ satisfying the conditions stated in that lemma.  For all $t\in \tilde{I}_2:=\cup\{I_j:j_-\leq j\leq j_+\}\subset\tilde{I}_1$ we have
$\dist(t,I_{j_2})\leq |\tilde{I}_{j_2}|\leq c^{-1}\eta^{-C}|I_{j_2}|$.  If $j_+-j_-\leq \tilde{C}\eta^{-C}$, the conclusion follows by setting $K=2$.  Otherwise, we repeats the argument with another application of Lemma $\ref{lem2}$.  This procedure can be iterated $K\gtrsim \eta^{C}\log J'$ times, until one of the constructed pairs $(j_-,j_+)$ satisfies $j_+-j_-\leq \tilde{C}\eta^{-C}$, or until no more intervals are left to remove.
\end{proof}

\subsection{Conclusion of the argument}

We now complete the proof of Proposition $\ref{prop-uniform}$.  Apply Proposition $\ref{prop-recursive}$ to choose $t_*$, $K$, and $I_{j_1},\cdots,I_{j_K}$ satisfying the stated properties.  Then, for each $k=1,\cdots,K$, we have,  by Corollary $\ref{cor-concentration}$,
\begin{align*}
M(u(t);0,C\eta^{-C}|I_{j_k}|^{1/2})\geq c\eta^{C}|I_{j_k}|^{7/12}.
\end{align*}
for all $t\in I_{j_k}$.  This gives, with $t\in I_{j_k}$,
\begin{align*}
&M(u(t_*);0,C\eta^{-C}|I_{j_k}|^{1/2})\\
&\hspace{0.2in}=M(u(t);0,C\eta^{-C}|I_{j_k}|^{1/2})-\int_t^{t_*} \partial_sM(u(s);0,C\eta^{-C}|I_{j_k}|^{1/2})ds\\
&\hspace{0.2in}\geq c\eta^{C}|I_{j_k}|^{7/12}-(|I_{j_k}|+\dist(t_*,I_{j_k}))C\eta^{5C/6}|I_{j_k}|^{-5/12}E\\
&\hspace{0.2in}\geq (c\eta^{C}-C\eta^{5C/6}E)|I_{j_k}|^{7/12}\\
&\hspace{0.2in}\geq c\eta^{C}|I_{j_k}|^{7/12}
\end{align*}
for suitable choices of the constants, while the mass bound (\ref{mass-bound}) gives
\begin{align*}
M(u(t_*);0,2C\eta^{-C}|I_{j_k}|^{1/2})&\lesssim \eta^{-7C/6}|I_{j_k}|^{7/12}E\lesssim \eta^{-C}|I_{j_k}|^{7/12}.
\end{align*}

Fix $1\leq k\leq K$.  Now, setting $B_k=B(0,C\eta^{-C}|I_{j_k}|^{1/2})$ and fixing $N\geq 1$, it follows that
\begin{align*}
&\int_{B_k\setminus \cup\{B_\ell:k+N\leq \ell\leq K\}} |u(t_*)|^2dx\\
&\hspace{0.2in}\geq \int_{B_k} |u(t_*)|^2dx-\sum_{\ell=k+N}^K \int_{B_\ell} |u(t_*)|^2dx\\
&\hspace{0.2in}\geq M(u(t_*);0,C\eta^{-C}|I_{j_k}|^{1/2})^2-\sum_{\ell=k+N}^K M(u(t_*);0,2C\eta^{-C}|I_{j_\ell}|^{1/2})^2\\
&\hspace{0.2in}\geq c\eta^C|I_{j_k}|^{7/6}-\sum_{\ell=k+N}^K C\eta^{-C}|I_{j_\ell}|^{7/6}\\
&\hspace{0.2in}\geq c\eta^C|I_{j_k}|^{7/6},
\end{align*}
provided $N=C\log(\eta^{-1})$ with $C$ sufficiently large, where we have used the construction of the sequence $(I_{j_k})$ to obtain the inequality $\sum_{\ell=k+N}^K |I_{j_\ell}|^{7/6}\leq 2^{-7N/12}|I_{j_k}|^{7/6}$.

This in turn gives
\begin{align*}
c\eta^C|I_{j_k}|^{7/6}&\leq C\eta^{-7C/3}|I_{j_{k+N}}|^{7/6}\int_{B_k\setminus\cup\{B_\ell:k+N\leq \ell\leq K\}} \frac{|u(t_*,x)|^2}{|x|^{7/3}}dx\\
&\lesssim \eta^{-7C/3}|I_{j_k}|^{7/6}\int_{B_k\setminus\cup \{B_\ell:k+N\leq \ell\leq K\}} \frac{|u(t_*,x)|^2}{|x|^{7/3}}dx
\end{align*}
where in the last inequality we have used the bound $|I_{j_{k+N}}|\leq |I_{j_k}|$.  

We now sum over values of $k$ in the set $\{k:1\leq k\leq K$ and $k=1+mN $ for some $m\geq 0\}$.  Since this set has size comparable to $K/N$, we obtain (via the Hardy inequality $\lVert u/|x|^\alpha\rVert_{L^2}\lesssim \lVert |\nabla|^{\alpha}u\rVert_{L^2}$ on $\mathbb{R}^3$ for $0\leq \alpha<3/2$)
\begin{align*}
c\eta^C(K/N)\leq \eta^{-7C/3}\int_{B_1} \frac{|u(t_*,x)|^2}{|x|^{7/3}}dx\lesssim \eta^{-7C/3}\int_{\mathbb{R}^3} ||\nabla|^{s_c}u(t_*)|^2dx\lesssim \eta^{-7C/3}E^2,
\end{align*}
i.e.
\begin{align*}
K\lesssim N\eta^{-C}\leq C\eta^{-C}.
\end{align*}
Recalling the bound $K\gtrsim \eta^C\log(J')$ then gives $\log(J')\lesssim \eta^{-C}$, so that $J'\leq \exp(C\eta^{-C})=\exp(CE^C)$.  Moreover, recalling the estimate $\#B\lesssim \eta^{-C}$ given by Remark $\ref{rem-count}$, we have $J\leq \exp(CE^C)$.  We therefore get
\begin{align*}
\int_{t_-}^{t_+}\int_{\mathbb{R}^3}|u(t,x)|^{15}dxdt=\sum_{j=1}^J \int_{I_j}\int_{\mathbb{R}^3}|u(t,x)|^{15}dxdt\leq 2J\eta\lesssim \exp(CE^C)
\end{align*}
which is the desired uniform bound.  This completes the proof of Proposition $\ref{prop-uniform}$.

\section{Proof of Theorem \ref{thm1}}

In this section we prove Theorem $\ref{thm1}$.

\begin{proof}[Proof of Theorem $\ref{thm1}$]
Fix $\delta>0$ to be determined later in the argument.  Let $u$ be a radial solution to (\ref{eq1}) on a time interval $I=[t_-,t_+]$.  

For each interval $J\subset I_{\textrm{max}}$, set 
\begin{align*}
S(u,J)&=\lVert u\rVert_{L_t^\infty(J;\dot{H}_x^{s_c}(\mathbb{R}^3))}+\lVert u\rVert_{L_t^\infty(J;\dot{H}_x^{s_c+1}(\mathbb{R}^3))}\\
&\hspace{0.4in}+\lVert u\rVert_{L_{t,x}^{15}(J\times\mathbb{R}^3)}
+\lVert |\nabla|^{s_c}u\rVert_{L_{t,x}^{10/3}(J\times \mathbb{R}^3)}+\lVert |\nabla|^{s_c+1}u\rVert_{L_{t,x}^{10/3}(J\times \mathbb{R}^3)},
\end{align*}
and for ease of notation set $S(u,T)=S(u,[0,T])$ for each $T>0$.  Fix a parameter $R_0>0$ to be determined later in the argument.  

We use a continuity argument to show that if $R_0$ is chosen to be sufficiently large then $S(u,T)\leq R_0$ for all $T$.  Suppose that $T>0$ is such that $S(u,T)\leq R_0$.  By the local theory, we can find $T'>T$ so that $S(u,T')\leq 2R_0$.  This, combined with our hypothesis on the $L_t^\infty\dot{H}_x^{s_c-\delta}$ norm, gives 
\begin{align*}
\lVert u\rVert_{L_t^\infty([0,T'];\dot{H}_x^{s_c}(\mathbb{R}^3))}&\leq \sup_{t\in [0,T']} \lVert u\rVert_{\dot{H}_x^{s_c-\delta}}^{1-\delta}\lVert u\rVert_{\dot{H}_x^{s_c+1-\delta}}^\delta\leq E_0^{1-\delta}(2R_0)^\delta,
\end{align*}
and thus, as a consequence of Proposition $\ref{prop-uniform}$,
\begin{align}
\lVert u\rVert_{L_{t,x}^{15}([0,T']\times\mathbb{R}^3)}\leq C\exp(CE_0^{C}R_0^{C\delta}),\label{2-eq-unif-1}
\end{align}
where we have set $E_0:=E$ with $E\geq 1$ as in ($\ref{eq-E0}$).

Now, by the Sobolev and Strichartz estimates, one has, for any $J=[t_1,t_2]\subset I$,
\begin{align}
\nonumber S(u,J)&\lesssim \lVert u(t_1)\rVert_{\dot{H}_x^{s_c}\cap \dot{H}_x^{s_c+1}}+\lVert |u|^6u\rVert_{N(J)}+\lVert |\nabla|[|u|^6u]\rVert_{N(J)}\\
\nonumber &\lesssim \lVert u(t_1)\rVert_{\dot{H}_x^{s_c}\cap \dot{H}_x^{s_c+1}}+\lVert |\nabla|^{s_c}u\rVert_{L_{t,x}^{10/3}(J\times\mathbb{R}^3)}\lVert u\rVert_{L_{t,x}^{15}(J\times\mathbb{R}^3)}^6\\
\nonumber &\hspace{0.2in}+\lVert |\nabla|^{s_c+1}u\rVert_{L_{t,x}^{10/3}(J\times\mathbb{R}^3)}\lVert u\rVert_{L_{t,x}^{15}(J\times \mathbb{R}^3)}^6\\
&\lesssim \lVert u(t_1)\rVert_{\dot{H}_x^{s_c}\cap \dot{H}_x^{s_c+1}}+2S(u,J)\lVert u\rVert_{L_{t,x}^{15}(I\times\mathbb{R}^3)}^6\label{eq-str-01}
\end{align}

Let $\tilde{C}\geq 1$ denote the implicit constant in the inequality (\ref{eq-str-01}).  Fix $\epsilon>0$ and partition the interval $[0,T]$ into intervals $I_1,I_2,\cdots,I_m$ with $$\lVert u\rVert_{L_{t,x}^{15}(I_j\times\mathbb{R}^3)}^6=\epsilon,\quad 1\leq j<m,$$ and $$\lVert u\rVert_{L_{t,x}^{15}(I_m\times\mathbb{R}^3)}^6<\epsilon.$$  Note that (\ref{2-eq-unif-1}) implies 
\begin{align}
m\leq (C/\epsilon)\exp(CE_0^C(2R_0)^{C\delta}).\label{2-eq-m}
\end{align}
Applying (\ref{eq-str-01}) to each of the intervals $I_j=[t_j,t_{j+1}]$, $1\leq j\leq m$ now gives $S(u,I_1)\leq 2\tilde{C}\lVert u(0)\rVert_{\dot{H}_x^{s_c}\cap \dot{H}_x^{s_c+1}}$ and
\begin{align*}
S(u,I_j)\leq 2\tilde{C}\lVert u(t_j)\rVert_{\dot{H}_x^{s_c}\cap \dot{H}_x^{s_c+1}}\leq 2\tilde{C}S(u,I_{j-1})
\end{align*}
for all $1<j\leq m$, provided $\epsilon\leq (4\tilde{C})^{-1}$.  We therefore get $S(u,I_j)\leq (2\tilde{C})^{j-1}S(u,I_1)$ for $1<j\leq m$, which in turn (combined with the above) gives $$S(u,I_j)\leq (2\tilde{C})^{j}\lVert u(0)\rVert_{\dot{H}_x^{s_c}\cap \dot{H}_x^{s_c+1}}$$ for all $1\leq j\leq m$.

Putting this together, we get 
\begin{align*}
S(u,T')&\leq (2\tilde{C})^m\lVert u_0\rVert_{\dot{H}_x^{s_c}\cap\dot{H}_x^{s_c+1}}+\bigg(\sum_{j=1}^m \lVert u\rVert_{L_{t,x}^{15}(I_j\times\mathbb{R}^3)}^{15}\bigg)^{1/15}\\
&\hspace{0.4in}+\bigg(\sum_{j=1}^m \lVert |\nabla|^{s_c}u\rVert_{L_{t,x}^{10/3}(I_j\times\mathbb{R}^3)}^{10/3}\bigg)^{3/10}+\bigg(\sum_{j=1}^m\lVert |\nabla|^{s_c+1}u\rVert_{L_{t,x}^{10/3}(I_j\times\mathbb{R}^3)}^{10/3}\bigg)^{3/10}\\
&\leq 4C'(2\tilde{C})^m\lVert u(0)\rVert_{\dot{H}_x^{s_c}\cap\dot{H}_x^{s_c+1}}
\end{align*}
for some $C'=C'(C)>0$.

Now, choose 
\begin{align*}
R_0\geq 4C'(2\tilde{C})^{(C/\epsilon)\exp(2CE_0^C)}\lVert u_0\rVert_{\dot{H}_x^{s_c}\cap \dot{H}_x^{s_c+1}}.
\end{align*}
and $\delta_0>0$ small enough to ensure 
\begin{align*}
(2R_0)^{C\delta}\leq 2.
\end{align*}
Recalling (\ref{2-eq-m}), we therefore obtain $S(u,T')\leq R_0$.  Since $T>0$ was an arbitrary value for which $S(u,T)\leq R_0$, this establishes the desired uniform bound.
\end{proof}

\noindent {\it Remark}. A close examination of the proof shows that the regularity threshold $s_c+1$ in the statement of Theorem $\ref{thm1}$ 
can be relaxed to $s_c+\epsilon$, as given in the statement below (with $\delta$ now dependent on $\epsilon$).

\begin{corollary}
There exists $C>0$ such that for each $E\geq 1$, $M>0$ and $\epsilon>0$ there exists $\delta_0>0$ such that for all $0<\delta<\delta_0$ and $0\in J\subset\mathbb{R}$, if 
$u\in C_t(J;\dot{H}_x^{s_c-\delta}(\mathbb{R}^3)\cap \dot{H}_x^{s_c}(\mathbb{R}^3))$ is a radially symmetric solution to (\ref{eq1}) which satisfies 
$$\lVert u_0\rVert_{\dot{H}_x^{s_c}(\mathbb{R}^3)\cap \dot{H}_x^{s_c+\epsilon}(\mathbb{R}^3)}\leq M,$$ and, 
$$\lVert u\rVert_{L_t^\infty(I_{\textrm{max}};\dot{H}_x^{s_c-\delta}(\mathbb{R}^3))}\leq E,$$
then $I_{\textrm{max}}=\mathbb{R}$ and $$\lVert u\rVert_{L_{t,x}^{15}(\mathbb{R}\times\mathbb{R}^3)}\leq C\exp(C(EM^\delta)^C).$$
\end{corollary}

\noindent We remark that elementary scaling considerations show that the restriction $\epsilon>0$ is essential, at least for a statement in this form.

\section{Proof of Corollary $\ref{cor1}$: allowing for growth}

In this section we complete the proof of Corollary $\ref{cor1}$.

\begin{proof}[Proof of Corollary $\ref{cor1}$]
For each interval $J\subset I_{\textrm{max}}$, set $$T(u,J)=\lVert u\rVert_{L_t^\infty(J;\dot{H}_x^{s_c}(\mathbb{R}^3))}+\lVert u\rVert_{L_{t,x}^{15}(J\times\mathbb{R}^3)}+\lVert |\nabla|^{s_c}u\rVert_{L_{t,x}^{10/3}(J\times\mathbb{R}^3)},$$ and for ease of notation set $T(u,t_+)=T(u,[0,t_+])$ for each $t_+>0$.  Fix a parameter $M_0>0$ to be determined later in the argument.  As in the proof of Theorem $\ref{thm1}$, we use a continuity argument to show that if $M_0$ is sufficiently large then $T(u,t_+)\leq M_0$ for all $t_+$.

Suppose that $t_+>0$ is such that $T(u,t_+)\leq M_0$.  By the local theory, we can find $t'>t_+$ so that $T(u,t')\leq 2M_0$.  This, combined with our hypothesis on the $L_t^\infty\dot{H}_x^{s_c}$ norm, gives $$\lVert u\rVert_{L_t^\infty([0,t'];\dot{H}_x^{s_c}(\mathbb{R}^3))}\leq g(2M_0),$$
and thus, as a consequence of Proposition $\ref{prop-uniform}$, 
\begin{align}
\lVert u\rVert_{L_{t,x}^{15}([0,t']\times\mathbb{R}^3)}\leq C\exp(C[g(2M_0)]^C).\label{eq-unif-1}
\end{align}

Now, by the Sobolev and Strichartz estimates, one has, for any $J=[t_1,t_2]\subset I$,
\begin{align}
\nonumber T(u,J)&\lesssim \lVert u(t_1)\rVert_{\dot{H}_x^{s_c}}+\lVert |u|^6u\rVert_{N(I)}\\
\nonumber &\lesssim \lVert u(t_1)\rVert_{\dot{H}_x^{s_c}}+\lVert |\nabla|^{s_c}u\rVert_{L_{t,x}^{10/3}(I\times\mathbb{R}^3)}\lVert u\rVert_{L_{t,x}^{15}(I\times\mathbb{R}^3)}^6\\
&\lesssim \lVert u(t_1)\rVert_{\dot{H}_x^{s_c}}+T(u,J)\lVert u\rVert_{L_{t,x}^{15}(I\times\mathbb{R}^3)}^6\label{eq-str-1}
\end{align}

Let $\tilde{C}$ denote the implicit constant in the inequality (\ref{eq-str-1}).  Fix $\epsilon>0$ and partition the interval $[0,t']$ into intervals $I_1,I_2,\cdots,I_m$ with $$\lVert u\rVert_{L_{t,x}^{15}(I_j\times\mathbb{R}^3)}^6=\epsilon,\quad 1\leq j<m,$$ and $$\lVert u\rVert_{L_{t,x}^{15}(I_m\times\mathbb{R}^3)}^6<\epsilon.$$  Note that (\ref{eq-unif-1}) implies 
\begin{align}
m\leq C\exp(C[g(2M_0)]^C)/\epsilon\leq (C/\epsilon)\log^{1/2}(2M_0).\label{eq-m}
\end{align}
Applying (\ref{eq-str-1}) to each of the intervals $I_j=[t_j,t_{j+1}]$, $1\leq j\leq m$ now gives $T(u,I_1)\leq 2\tilde{C}\lVert u(0)\rVert_{\dot{H}_x^{s_c}}$ and
\begin{align*}
T(u,I_j)\leq 2\tilde{C}\lVert u(t_j)\rVert_{\dot{H}_x^{s_c}}\leq 2\tilde{C}\lVert u\rVert_{L_t^\infty(I_{j-1};\dot{H}_x^{s_c})}\leq 2\tilde{C}T(u,I_{j-1})
\end{align*}
for all $1<j\leq m$, provided $\epsilon\leq (2\tilde{C})^{-1}$ (in what follows, we make the choice $\epsilon=(2\tilde{C})^{-1}$).  We therefore get $T(u,I_j)\leq (2\tilde{C})^{j-1}T(u,I_1)$ for $1<j\leq m$, which in turn (combined with the above) gives $$T(u,I_j)\leq (2\tilde{C})^{j}\lVert u(0)\rVert_{\dot{H}_x^{s_c}}$$ for all $1\leq j\leq m$.

Putting this together, we get
\begin{align*}
T(u,t')&\leq (2\tilde{C})^m\lVert u_0\rVert_{\dot{H}_x^{s_c}}+\left(\sum_{j=1}^m \lVert u\rVert_{L_{t,x}^{15}(I_j\times\mathbb{R}^3)}^{15}\right)^{1/15}\\
&\hspace{1.2in}+\left(\sum_{j=1}^m \lVert |\nabla|^{s_c}u\rVert_{L_{t,x}^{10/3}(I_j\times\mathbb{R}^3)}^{10/3}\right)^{3/10}\\
&\leq 3C'(2\tilde{C})^m\lVert u(0)\rVert_{\dot{H}_x^{s_c}}
\end{align*}
for some $C'=C'(C)>0$.

To conclude, recall (\ref{eq-m}) and note that choosing $M_0$ large enough to ensure $(C/\epsilon)\log^{1/2}(2M_0)=2C\tilde{C}\log^{1/2}(2M_0)\leq \log(M_0/(3C'\lVert u(0)\rVert_{\dot{H}_x^{s_c}}))/\log(2C)$ therefore gives
\begin{align*}
3C'(2\tilde{C})^m\lVert u(0)\rVert_{\dot{H}_x^{s_c}}\leq M_0,
\end{align*}
so that $T(u,t')\leq M_0$ (to see that this choice of $M_0$ is possible, note that for every $\lambda>0$, $\log^{1/2}(t)/\log(t/\lambda)\rightarrow 0$ as $t\rightarrow\infty$).

Since $t_+>0$ was an arbitrary value for which $T(u,t_+)\leq M_0$, this establishes the desired uniform bound.
\end{proof}

\end{document}